\documentclass[11pt]{article}

\usepackage{algorithm, algpseudocode}
\usepackage{amsmath,amssymb,amsthm}
\usepackage{mathrsfs}
\usepackage[bookmarks=true,colorlinks,linkcolor=blue]{hyperref}
\usepackage{subfig}
\usepackage{color} 
\usepackage{graphicx}
 \vspace{0.6cm}

\newtheorem{lemma}{Lemma}[section]
\newtheorem{Defi}{Definition}[section]
\newtheorem{them}[lemma]{Theorem}

\newtheorem{rem}{Remark}

\newtheorem{pro}{Problem}
\DeclareMathOperator{\rank}{rank}
\DeclareMathOperator{\res}{Res}
\allowdisplaybreaks

\begin{document}

\title{\textbf{
Square-free Discriminants of Matrices
and the Generalized Spectral Characterizations of Graphs}\footnote{This work is supported by National Natural Science Foundation of China (No. 11471005)}}
\author{
\small Wei Wang \footnote{ E-mail
address: wang$\_$weiw@163.com. }
 \qquad Tao Yu
\\  \small School of Mathematics and Statistics,
 Xi'an Jiaotong University,\\
\small  No.28 Xianning West Road, Xi'an, 710049, P.R. China}
\date{}
\maketitle


\abstract \vspace{0.5cm} Let $S_n(\mathbb{Z})$ and $O_n(\mathbb{Q})$ denote the set of all $n\times n$ symmetric matrices over the ring of integers $\mathbb{Z}$ and the set of all $n\times n$ orthogonal matrices over the field of rational numbers $\mathbb{Q}$, respectively. The paper is mainly concerned with the following problem: Given a matrix $A\in {S_n(\mathbb{Z})}$. How can one find all rational orthogonal matrices $Q\in{O_n(\mathbb{Q})}$ such that $Q^TAQ\in {S_n(\mathbb{Z})}$, and in particular, when does $Q^TAQ\in {S_n(\mathbb{Z})}$ with $Q\in{O_n(\mathbb{Q})}$ imply that $Q$ is \emph{a signed permutation matrix} (i.e., the matrix obtained from a permutation matrix $P$ by replacing each 1 in $P$ with 1 or $-1$)?

A surprisingly simple answer was given in terms of whether the discriminant of the characteristic polynomial of $A$ is odd and square-free, which partially answers the above questions. More precisely, let $\Delta_A=\pm \res(\phi,\phi')$ be \emph{the discriminant of matrix $A$}, where $\res(\phi,\phi')$ is \emph{the resultant} of the characteristic polynomial $\phi$ of $A$ and its derivative $\phi'$. We show that if $\Delta_A$ is odd and square-free, then $Q^TAQ\in {S_n(\mathbb{Z})}$ with $Q\in{O_n(\mathbb{Q})}$ implies that $Q$ is a signed permutation matrix. As an application, we present a simple and efficient method for testing whether a graph is determined by the generalized spectrum, which significantly extends our previous work.\\

\noindent
Key Words: Discriminant; Square-free; Resultant; Rational orthogonal matrix; Spectra of graphs; Cospectral graphs; Determined by spectrum\\
\noindent
AMS classification: 05C50

\newpage

\section{Introduction}
Throughout the paper, let $\mathbb{Z}$, $\mathbb{Q}$ and $\mathbb{F}_p$ denote the ring of integers, the field of rational numbers and the finite field with $p$ (a prime number) elements, respectively. Let $S_n(\mathbb{Z})$ and $O_n(\mathbb{Q})$ denote the set of all symmetric matrices of order $n$ over $\mathbb{Z}$ and the set of all orthogonal matrices of order $n$ over $\mathbb{Q}$, respectively.

A square matrix $P$  is called \emph{a singed permutation matrix} if each of its row  and column has exactly one non-zero entry $\pm 1$. In this paper, we are mainly concerned with the following:
\begin{pro} Given a matrix $A\in {S_n(\mathbb{Z})}$, how can one find all rational orthogonal matrices $Q\in{O_n(\mathbb{Q})}$ such that $Q^TAQ\in {S_n(\mathbb{Z})}$?
\end{pro}
Clearly, if $Q$ is a signed permutation matrix, $Q^TAQ$ is an integral matrix for any $A\in {S_n(\mathbb{Z})}$. So in particular, we are more interested in
\begin{pro} When does $Q^TAQ\in {S_n(\mathbb{Z})}$ with $Q\in{O_n(\mathbb{Q})}$ imply that $Q$ is a signed permutation matrix?
\end{pro}

The motivation of our study comes from the recent work of Wang~\cite{W3,W4}, Wang and Xu~\cite{W1,W2}, in an attempt to characterize graphs by their generalized spectra, which will be briefly described below.

Let $G$ be a simple graph with adjacency matrix $A$. \emph{The spectrum} of $G$, denoted by $\sigma{(G)}$, consists of the eigenvalues (including multiplicities) of the matrix $A$. The spectrum of $G$ together with that of its complement, denoted by $\sigma(G)\& \sigma(\bar{G})$, is referred as \emph{the generalized spectrum} of $G$ in the paper. Two graphs $G$ and $H$ are \emph{cospectral} if they share the same spectrum (i.e., $\sigma(G)=\sigma(H)$). In particular, $G$ and $H$ are \emph{cospectral w.r.t. the generalized spectrum} if $\sigma(G)=\sigma(H)$ and $\sigma(\bar{G})=\sigma(\bar{H})$. A graph $G$ is said to be \emph{determined by the spectrum}, DS for short (resp. \emph{determined by the generalized spectrum}, DGS for short), if for any graph $H$, $\sigma(G)=\sigma(H)$ (resp. $\sigma(G)=\sigma(H)$ and $\sigma(\bar{G})=\sigma(\bar{H})$) implies that $H$ is isomorphic $G$.

 A fundamental question in the theory of graph spectra is: ``What kinds of
graphs are determined by the spectrum?" The problem dates back
to more than 50 years ago and originates from chemistry, which is also closely related to a famous problem of Kac~\cite{Kac}:  ``Can one hear the shape
of a drum?" For the background and some known results about this problem, we refer the reader to \cite{DH,DH1} and the references therein.

 Whereas it is comparatively easy to construct pairs of cospectral but non-isomorphic graphs, it is more challenging to show a given graph or a family of graphs are DS. Recently, Wang and Xu~\cite{W1,W2}, Wang~\cite{W3,W4} developed a powerful method to test whether a controllable graph is DGS. Here a graph is \emph{controllable} means its walk-matrix $W=[e, Ae,\cdots, A^{n-1}e]$ ($e$ is all-ones vector; $n$ is the order of the graph) is non-singular. Denote by ${\cal{G}}_n$ the set of all controllable graphs. It was conjectured by Godsil \cite{G} that almost every graph is controllable. A proof of this conjecture has been recently announced in \cite{OT}.

 The key observation of the method of Wang and Xu is the following

\begin{them}[Wang and Xu~\cite{W1}] Let $G \in {\cal G}_n$. Then there exists a graph $H$ that is cospectral with $G$ w.r.t. the generalized spectrum if and only if there exists a rational
orthogonal matrix $Q$ such that $Q^{\mathrm{T}}A(G)Q = A(H)$ and $Qe = e$.
\end{them}

Let $\mathcal{Q}_G=\{Q\in{O_n(\mathbb{Q})}|Q^TAQ ~{\rm is~a~}(0,1)-{\rm matrix ~and}~ Qe=e\}$. We have the following
\begin{them}[Wang and Xu~\cite{W1}] Let $G \in {\cal G}_n$. Then $G$ is DGS
if and only if the set ${\cal Q}_G$ contains only permutation matrices.
\end{them}
By the theorem above, for a controllable graph $G$ with adjacency matrix $A$, in order to tell whether $G$ is DGS or not, we have to deal with the following
\begin{pro} How can one find all $Q\in{O_n(\mathbb{Q})}$ with $Qe=e$ such that if $Q^TAQ$ is a $(0,1)$-matrix? and in particular, when does the fact that $Q\in{O_n(\mathbb{Q})}$ with $Qe=e$ such that $Q^TAQ$ is a $(0,1)$-matrix imply that $Q$ is a permutation matrix?
\end{pro}

Form the above discussions, we can be see that Problems 1 and 2 are natural extensions of Problem 3, and moreover, the resolution of Problems 1 and 2 would help
us in solving Problem 3, as we shall see later. We would like to mention that, to the best of our knowledge, we know no article in the literatures dealing with Problems 1 and 2, and only a slightly related problem was considered by Friedland~\cite{Fr}: Given two rational symmetric matrices with the same characteristic polynomial, when does there exist a $Q\in{O_n(\mathbb{Q})}$ such that $B=Q^TAQ$?

In Wang~\cite{W3,W4}, the author gave a simple arithmetic criterion to solve Problem 3, in terms of the pattern of the prime factorization of $\det(W)$. More precisely, we have
\begin{them}[Wang~\cite{W3,W4}] \label{Main}If $\frac{\det(W)}{2^{\lfloor\frac{n}{2}\rfloor}}$ (which is always an integer) is odd and square-free, then ${\cal Q}_G$ contains only permutation matrices, and hence $G$ is DGS.
\end{them}

We found that, somewhat surprisingly, there exists an analogous result for solving Problem 2. Denote by $\Delta_A=\pm \res(\phi,\phi')$ \emph{the discriminant of matrix $A$}, where $ \res(\phi,\phi')$ is the resultant of polynomial $\phi$ and its derivative $\phi'$ with $\phi$ being the characteristic polynomial of matrix $A$ (see Section 2 for details). The main result of the paper is the following
\begin{them}\label{Main1}
Give a matrix $A\in{S_n(\mathbb{Z})}$. If $\Delta_A$ is odd and square-free, then $Q^TAQ\in S_n(\mathbb{Z})$ with $Q\in{O_n(\mathbb{Q})}$ implies that $Q$ is a signed permutation matrix.
\end{them}
\vspace{1mm}
\begin{rem} In a recent break-through work~\cite{BSW}, Bhargava et. al. proved that the density of polynomials with square-free discriminants is approximately $35.8232\%$; see Section 7 for more details. We conjecture that a similar result holds for symmetric integral matrices, which would imply that Theorem~\ref{Main1} holds for a positive fraction of matrices $A\in{S_n(\mathbb{Z})}$.
\end{rem}

As a consequence of Theorems~\ref{Main} and \ref{Main1}, we have

\begin{them}\label{Main2} Let $G$ be a graph with adjacency matrix $A$. Let $d=\gcd(\frac{\det(W)}{2^{\lfloor\frac{n}{2}\rfloor}},\Delta_A)$. If $d$ is odd and square-free, then $G$ is DGS.
\end{them}
\vspace{1mm}
\begin{rem}Theorems~\ref{Main1} is the best possible in the sense that if we allow $p^2|\Delta_A$ for some odd $p$, then we cannot guarantee Theorems~\ref{Main1} to be true any more.
\end{rem}

\begin{rem}Compared with Theorem~\ref{Main}, Theorem~\ref{Main2} has more wide applicability, as we shall see in Section 5. Moreover, it saves more computational costs using Theorem~\ref{Main2} to test whether $G$ is DGS, since $\gcd(\frac{\det(W)}{2^{\lfloor\frac{n}{2}\rfloor}},\Delta_A)$ is much easier to compute than testing whether $\frac{\det(W)}{2^{\lfloor\frac{n}{2}\rfloor}}$ is square-free
or not, say, by factoring it (since there is no better way to tell a large number is square-free or not).
\end{rem}

The rest of the paper is organized as follows. In Section 2, we give some preliminary results that will be needed in the proof of Theorem~\ref{Main1}. In Section 3, we give a relation of the level of $Q$ and the discriminant of the matrix $A$. In Section 4, we present the proof
of Theorem~\ref{Main1}. In Section 5, we present the proof
of Theorem~\ref{Main2}. In Section 6, we give some examples to illustrate our method and then conduct some numerical experiments to compare Theorem~\ref{Main} and Theorem~\ref{Main2}. Conclusions are given in Section 7.

\section{Preliminaries}

For the convenience of the reader, in this section, we recall some notions and basic facts that will be needed later in the paper. Throughout, $p$ denotes a prime number and $\rank_p{(M)}$ denotes the rank of an integral $M$ over $\mathbb{F}_p$. We shall use $a\equiv b~({\rm mod}~p)$ and $a=b$ over $\mathbb{F}_p$ interchangeably.

\subsection{The level of a rational matrix}
The notion ``level" of a rational orthogonal matrix is proved to be useful in this paper.
\begin{Defi} The level of a rational orthogonal matrix $Q\in{O_n(\mathbb{Q})}$, denoted by $\ell(Q)$ or simply $\ell$, is the smallest positive integer $\ell$ such that $\ell Q $ is an integral matrix.
\end{Defi}

Clearly, if $\ell(Q)=1$, then $Q$ is a signed permutation matrix. Moreover, if $\ell(Q)=1$ and $Qe=e$, then $Q$ is a permutation matrix.

Recall that an $n$ by $n$ matrix $U$ with integer entries is called \emph{unimodular} if
det($U$) $= \pm1$. The following theorem is well
known.

\begin{them}
For every integral matrix $M$ with full rank, there exist unimodular
matrices $U$ and $V$ such that $M = USV$, where $S= diag(d_1, d_2, \cdots, d_n)$ is a diagonal matrix with $d_i$ being the $i$-th entry in the diagonal and $d_i\mid d_{i+1}$ for $i = 1, 2, \dots, n-1$.
\end{them}

For an integral matrix $M$, the above $S$ is called \emph{the Smith Normal Form} (SNF in short) of $M$, and $d_i$ is called the $i$-th \emph{elementary divisor} of the matrix $M$. It is noticed that the SNF of a matrix can be computed efficiently (see e.g. page 50 in~\cite{15}).

The following lemma plays a key role in the proof of Theorem~\ref{Main1}.

\begin{lemma}\label{Square}Let $M=Udiag(d_1,d_2,\cdots,d_n)V=USV$, where $S$ is the Smith Normal Form of
$M$, $U$ and $V$ are unimodular matrices and $d_i|d_{i+1}$ for
$i=1,2,\cdots,n-1$.  Then the system of congruence equations
$Mx\equiv 0~({\rm mod}~p^2)$ has a solution $x\not\equiv 0~({\rm mod}~p)$ if
and only if $p^2|d_n$,
\end{lemma}

\subsection{Some results on the generalized spectral characterization of graphs}

In this subsection, we review some previous results on the generalized spectral characterizations of graphs. As mentioned before, given a controllable graph $G$, the in order to tell
that $G$ is DGS or not, we have to find out all rational orthogonal matrices in ${\cal{Q}}_G$ explicitly. In particular, $G$ is DGS if and only if ${\cal{Q}}_G$ contains only permutation matrices. To achieve this goal, the basic strategy is to show that $\ell(Q)=1$ for every $Q\in{{\cal{Q}}_G}$.

Given a graph $G$, the following result shows that the prime divisor $p$ of $\ell$ cannot be arbitrary; it is always a divisor of $\det(W)$.

\begin{lemma}[Wang~\cite{W1}]\label{L} Let $G \in {\cal G}_n$. Let $Q \in {\cal Q}_G$ with level $\ell$, and $p$ be any prime. If $p|\ell$, then $p|d_n$ and hence $p|\det(W)$, where
$d_n$ is the $n$-th elementary divisor of $W$.
\end{lemma}

The following theorem shows that every odd prime factor of $\det(W)$ cannot be a divisor of $\ell$, whenever it is a prime divisor with multiplicity one.

\begin{them}[Wang~\cite{W3}]\label{LL} Let $G \in {\cal G}_n$. Let $Q \in {\cal Q}_G$ with level $\ell$, and $p$ be an odd
prime. If $p \mid \mathrm{det}(W)$ and $p^{2} \nmid \mathrm{det}(W)$, then $p\nmid\ell$.
\end{them}

The following theorem gives a simple condition, under which the prime $p=2$ is not a divisor of $\ell$.
\begin{them}[Wang~\cite{W4}] \label{LL1}Let $G \in {\cal G}_n$. Let $Q \in {\cal Q}_G$ with level $\ell$. Suppose that $2^{\lfloor n/2\rfloor+1}\nmid\det(W)$. Then $2\nmid \ell$.
\end{them}

Combining Theorems~\ref{LL} and \ref{LL1} together, it is easy to give a short proof of Theorem~\ref{Main}. Suppose that $\det(W)/2^{\lfloor n/2 \rfloor}$ is odd and square-free. Let $Q \in {\cal Q}_G$ with level $\ell$. If $\ell\neq 1$, let $p$ be any prime divisor of $\ell$. If $p=2$, according to Theorem~\ref{LL1}, we have $2\nmid \ell$; a contradiction. If $p$ is an odd prime, by Theorem~\ref{LL}, $p\nmid\ell$; a contradiction. Thus, $\ell=1$ and $Q$ is a permutation matrix. Thus, ${\cal{Q}}_G$ contains only permutation matrices and $G$ is DGS.

\subsection{Some facts about the discriminant and the resultant}
\label{secDiscr}

Let $f(x)$ and $g(x)$ be two polynomials over $\mathbb{Z}$. It is well known that
whether $f$ and $g$ have a common factor of degree larger than zero is closely related to \emph{the resultant} of $f$ and $g$.

Suppose that $\gcd(f(x),g(x))=d(x)$ with $\deg(d)>0$. Then there exist two polynomials $u(x)$ and $v(x)$ such that $f(x)=v(x)d(x)$ and $g(x)=u(x)d(x)$, where
$\deg(u)<\deg(g)$ and $\deg(v)<\deg(f)$. Thus, we have
\begin{equation}\label{EQ}
u(x)f(x)=v(x)g(x).
\end{equation}

Let
\begin{align*}
f(x)&=a_0x^n+a_1x^{n-1}+\cdots+a_{n-1}x+a_n,\\
g(x)&=b_0x^m+b_1x^{m-1}+\cdots+b_{m-1}x^x+b_m,\\
u(x)&=u_0x^{m-1}+u_1x^{m-2}+\cdots+u_{m-2}x+u_{m-1},\\
v(x)&=v_0x^{n-1}+v_1x^{n-2}+\cdots+v_{n-2}x+v_{n-1}.
\end{align*}
Comparing the coefficients of Eq.~(\ref{EQ}) gives $M^T\eta=0$, where
\[M=\begin{bmatrix}
a_0&a_1&\cdots&\cdots&\cdots&a_n\\
&a_0&a_1&\cdots&\cdots&\cdots&a_n\\
&&\cdots&\cdots&\cdots&\cdots\\
&&&a_0&a_1&\cdots&\cdots&\cdots&a_n\\
b_0&b_1&\cdots&\cdots&b_{m}\\
&b_0&b_1&\cdots&\cdots&b_{m}\\
&&\cdots&\cdots&\cdots&\cdots\\
&&&&b_0&b_1&\cdots&\cdots&b_{m}
\end{bmatrix}\]
is the Sylvester matrix and $\eta =(u_0,u_1,\cdots,u_{n-2},-v_0,-v_1,\cdots,-v_{n-1})^T$. Clearly, if $\eta\neq 0$, then $\det(M)=0$.
Thus, we reach the following definition:
\begin{Defi} Let $f$ and $g$ be non-zero polynomials over some field $\mathbb{K}$. Then the resultant of $f$ and $g$, denoted by $\res(f,g)$, is defined to be the determinant of the Sylvester matrix $M$.
\end{Defi}

\begin{Defi}Let $f(x)=x^n+a_1x^{n-1}+\cdots +a_{n-1}x+a_n$ be a polynomial over $\mathbb{Z}$. The discriminant of $f$ is defined to be
$$\Delta(f)=\prod_{i<j}(\alpha_i-\alpha_j)^2,$$
\end{Defi}
where $\alpha_i$'s are the roots of $f$ over $\mathbb{C}$.

\begin{Defi}Let $\phi(x)=\det(x I-A)$ be the characteristic polynomial of matrix $A$. The discriminant of $A$, denoted by $\Delta_A$, is defined to be ${\Delta}(\phi)$.
\end{Defi}

There is a close relation between the discriminant of a polynomial $f$ and the resultant of $f$ and $f'$.

\begin{them}[see e.g., Lang~\cite{Lang}] Let $f(x)$ be a polynomial with leading coefficient $a_0=1$. Then $\Delta(f)=(-1)^{n(n-1)/2}\res(f,f')$.
\end{them}

Thus, we have $\Delta_A=\pm \res(\phi,\phi')$, which provides a better way to compute $\Delta_A$ through computing the determinant of the Sylveter matrix associated with $\phi$.
The following Theorem will be used frequently in the proof of Theorem~\ref{Main1}.
\begin{them}[see e.g., Lang~\cite{Lang}]\label{divisor} $f(x)\in{\mathbb{Z}[x]}$ has multiple factors over $\mathbb{F}_p$ if and only if $p|\res(f,f')$, or equivalently $\res(f,f')=0$ over $\mathbb{F}_p$.
\end{them}

\section{The level of $Q$ and the discriminant $\Delta_A$}

Given $A\in{S_n(\mathbb{Z})}$. Let $Q\in{O_n(\mathbb{Q})}$ with level $\ell$ such that $Q^TAQ=B\in S_n(\mathbb{Z})$. Let $p$ be any prime factor of $\ell$. The main objective of this section is to show that $p$ is always a divisor of $\Delta_A$.

Note that $\bar{Q}=\ell Q$ is an integral matrix. We can assume that $q_1,q_2\cdots,q_s$ to be a maximal independent system of the column vectors of $\bar{Q}$ modulo $p$. It follows from $A\bar{Q}=\bar{Q}B$ that
\begin{equation}\label{MM}
Aq_i=\tilde{b}_{i1}q_1+\tilde{b}_{i2}q_2+\cdots+\tilde{b}_{is}q_s,~i=1,2,\cdots,s,
\end{equation}
where $\tilde{B}:=(\tilde{b}_{ij})$ is an $s$ by $s$ matrix over the finite field $\mathbb{F}_p$. Note that $Q$ is an orthogonal matrix, we obtain
\begin{equation}\label{GG}
q_i^Tq_j=0~\mbox{over}~\mathbb{F}_p,~i,j=1,2,\cdots,s.
\end{equation}

\begin{lemma}\label{L01}
Assume that Eqs. (\ref{MM}) and (\ref{GG}) hold. Let $\psi(x)$ be the characteristic polynomial of matrix $\tilde{B}$ over $\mathbb{F}_p$. Then $\rank_p(\psi(A)^2)\leq n-s-1$.
\end{lemma}

\begin{proof}
Note that $A[q_1,q_2,\cdots,q_s]=[q_1,q_2,\cdots,q_s]B$. It follows that
\[\psi(A)[q_1,q_2,\cdots,q_s]=[q_1,q_2,\cdots,q_s]\psi(\tilde{B})=O.\]
This is because $\psi(\tilde{B})=O$ according to Cayley-Hamilton Theorem. Thus, we get $\psi(A)q_i=O$ for $i=1,2,\cdots,s$. In view of Eq.~(\ref{GG}), this implies that
\begin{equation}\label{KML}
q_i^T[\psi(A),q_1,q_2,\cdots,q_s]=O,{\rm for}~i=1,2,\cdots,s.
\end{equation}
Next, we distinguish the following two cases:

\textbf{Case 1.} There exists some $q_i$, say $q_1$, which can be expressed as the linear combination of the column vectors of the matrix $\psi(A)$, i.e., $q_1=\psi(A)\xi_1$ for some $0\neq \xi_1\in{\mathbb{F}_p^n}$. We claim that $\xi_1,q_i~(i=1,2,\cdots,s)$ are linearly independent over $\mathbb{F}_p$. For contradiction, suppose that there exist some constants $l_1,c_i\in{\mathbb{F}_p}$ that are not equal to zero simultaneously such that
\begin{equation}\label{PPO}
\sum_{k=1}^s c_kq_k+l_1\xi_1=0.
\end{equation}
Left multiplying both sides of Eq. (\ref{PPO}) by $\psi(A)$ gives
\[\sum_{k=1}^sc_k\psi(A)q_k+l_1\psi(A)\xi_1=l_1\psi(A)\xi_1=l_1q_1=0.\]
It follows that $l_1=0$. Then by Eq.~(\ref{PPO}) we get $c_k=0$ for $k=1,2,\cdots,s$; a contradiction. Thus, $q_1,\xi_i~(i=1,2,\cdots,s)$ are linearly independent. Note that $\psi(A)^2\xi_1=0,\psi(A)^2q_i=0~(i=1,2,\cdots,s)$, the conclusion that $\rank_p(\psi(A)^2)\leq n-s-1$ follows immediately by  the linearly independence of $q_1,\xi_i~(i=1,2,\cdots,s)$.

\textbf{Case 2}. Suppose none of $q_i$ can be expressed as a linear combinations of the column vectors of the matrix $\psi(A)$, for $i=1,2,\cdots,s$. Then we have
\[\rank_p(\psi(A))+1\leq \rank_p([\psi(A),q_1,q_2,\cdots,q_s]).\]
Moreover, it follows from Eq. (\ref{KML}) that $\rank_p([\psi(A),q_1,q_2,\cdots,q_s])\leq n-s$. Therefore, we have $\rank_p(\psi(A)^2)\leq \rank_p(\psi(A))\leq n-s-1$.

Combining Cases 1 and 2, the lemma follows. This completes the proof.
\end{proof}

\begin{lemma}\label{XS}
Given $A\in{S_n(\mathbb{Z})}$. Let $Q\in{O_n(\mathbb{Q})}$ with level $\ell$ such that $Q^TAQ=B\in S_n(\mathbb{Z})$. Let $p$ be any prime factor of $\ell$. Then the
characteristic polynomial $\phi(x)$ of the matrix $A$ must have a
multiple factor over $\mathbb{F}_p$.
\end{lemma}

\begin{proof} Since $q_1,q_2,\cdots,q_s$ are linearly independent over $\mathbb{F}_p$, there exist vectors $u_1,u_2,\cdots,u_{n-s}$ such that $q_i,u_{n-i}~(i=1,2,\cdots,s)$ form a basis for $\mathbb{F}_p^n$. Write $U=[q_1,\cdots,q_s,u_1,\cdots,u_{n-s}].$ Then by Eq. (\ref{MM}), we have
\begin{equation}\label{GM}
AU=U\begin{bmatrix}\tilde{B}&C\\O&D\end{bmatrix},
\end{equation}
where $C$ is an $s$ by $n-s$ matrix, and $D$ is an $n-s$ by $n-s$ matrix, over $\mathbb{F}_p$. Note that $U$ is non-singular, it follows from Eq. (\ref{GM}) that
\begin{equation}\label{GM1}
U^{-1}\psi(A)^2U=\begin{bmatrix}O&C^{*}\\O&\psi(D)^2\end{bmatrix},
\end{equation}
where we have used Cayley-Hamilton Theorem which implies that $\psi(\tilde{B})=O$, with $\psi(x)$ being the characteristic polynomial of the matrix $\tilde{B}$. By Eq.~(\ref{GM}), we have $\phi(x)=\psi(x)\chi(x)$ with $\chi(x)$ being the characteristic polynomial of the matrix $D$. If $\gcd(\psi(x),\chi(x))\neq 1$, then clearly $\phi(x)$ has a multiple factor. If not, there exist polynomials $h_1(x),h_2(x)\in{\mathbb{F}_p}[x]$ such that $h_1(x)\psi(x)+h_2(x)\chi(x)=1$. Using the fact that $\chi(D)=O$, we have $h_1(D)\psi(D)=I_{n-s}$, and hence $\rank_p(\psi(D)^2)=n-s$. It follows from Eq.~(\ref{GM1}) that $\rank_p(\psi(A)^2)=\rank_p(\psi(D)^2)=n-s$, which contradicts Lemma \ref{L01}. Thus $\phi(x)$ must have a multiple factor. This completes the proof.
\end{proof}

The following theorem shows that every prime divisor of $\ell$ is a divisor of $\Delta_A$.
\begin{them}\label{ellRes}
Given $A\in{S_n(\mathbb{Z})}$. Suppose that $Q^TAQ\in S_n(\mathbb{Z})$ for some $Q\in{O_n(\mathbb{Q})}$ with level $\ell$. Let $p$ be any prime factor of $\ell$. Then $p|\Delta_A$.
\end{them}
\begin{proof}This follows from Lemma~\ref{XS} and Theorem~\ref{divisor} directly.
\end{proof}

The following lemma plays a key role in the proof of Theorem~\ref{Main1}.
\begin{lemma}\label{LLX}
Let $Q$ be a rational orthogonal matrix with level $\ell$ such that $Q^TAQ=B\in{S_n(\mathbb{Z})}$. Let $p$ be a prime factor of $\ell$. Suppose that \begin{eqnarray}
\phi(x)=(x-\lambda_0)^2\varphi(x),~{\rm over}~\mathbb{F}_p,
 \end{eqnarray}
 where $\varphi(x)$ is square-free over $\mathbb{F}_p$. Let $q$ be an eigenvector of $A$ associated with $\lambda_0$ over $\mathbb{F}_p$, i.e., $Aq\equiv \lambda_0q~({\rm mod}~p$). The we have $q^T(Aq-\lambda_0q)\equiv 0~({\rm mod}~p^2)$.
\end{lemma}
\begin{proof}
By the proof of Lemma~\ref{XS}, we have $x-\lambda_0|\psi(x)$, where $\psi(x)$ is the characteristic polynomial of $\tilde{B}$. Let $q_1,q_2,\cdots,q_s$ be a system of maximal independent vectors of columns of $\ell Q$ (mod $p$). Then $V:=span\{q_1,q_2,\cdots,q_s\}$ is an invariant subspace of $A$. Note that the restriction of $A$ on $V$, $A|_V$, has an eigenvalue $\lambda_0$ with a corresponding eigenvector $q$. It follows that $q$ is the linear combinations of $q_1,q_2,\cdots,q_s$, i.e., there exist constants $c_1,c_2,\cdots,c_s$ such that $q=c_1q_1+c_2q_2+\cdots+c_kq_k$ over $\mathbb{F}_p$, or equivalently, $q=c_1q_1+c_2q_2+\cdots+c_sq_s+p\beta$ over $\mathbb{Z}$ for some integral vector $\beta$. Let $\bar{q}_1,\bar{q}_2,\cdots,\bar{q}_s$ be the columns of $\ell Q$ corresponding to $q_1,q_2,\cdots,q_s$. Then
\begin{equation}\label{OPQ}
q=k_1\bar{q}_1+k_2\bar{q}_2+\cdots+k_s\bar{q}_k+p\hat{\beta},
\end{equation}
over $\mathbb{Z}$ for some integral vector $\hat{\beta}$ and integers $k_1,k_2,\cdots,k_s$.

Next, we show $q^T(Aq-\lambda_0q)\equiv 0$ (mod $p^2$) holds. It follows from Eq.~(\ref{OPQ}) that
\begin{equation}\label{POP}
(q-p\hat{\beta})^T(q-p\hat{\beta})=\sum_{i,j}k_ik_j\bar{q}_i^T\bar{q}_j\equiv 0~({\rm mod}~p^2),
\end{equation}
where the congruence equation follows from $\bar{Q}^T\bar{Q}=\ell^2I$, which implies that $\bar{q}_i^T\bar{q}_j=\ell^2$ if $i=j$, and $0$ otherwise. By Eq.~(\ref{POP}) we get
\begin{equation}\label{XxX}
q^Tq-2p\hat{\beta}^Tq \equiv 0~({\rm mod}~p^2),
\end{equation}

Therefore, we have
\begin{align*}
q^TAq-\lambda_0q^Tq&=(\sum_{i,j}k_ik_j\bar{q}_i^TA\bar{q}_j+2p\bar{\beta}^TA\sum_{i}k_i\bar{q}_i+p^2\hat{\beta}^T\hat{\beta})-\lambda_0q^Tq\\
&\equiv 2p\hat{\beta}^TA\sum_{i}k_i\bar{q}_i-2p\lambda_0\hat{\beta}^Tq ~({\rm mod}~p^2),
\end{align*}
where we have used Eq.~(\ref{XxX}) and the equations $\bar{q}_i^TA\bar{q}_j=\ell^2b_{ij}\equiv 0$ (mod $p^2$) with $b_{ij}$ being the $(i,j)$-th entry of $B$, which follows from the fact that $\bar{Q}^TA\bar{Q}=\ell^2 B$.

Moreover, it follows from Eq.~(\ref{OPQ}) that $A\sum_{i}k_i\bar{q}_i=Aq-pA\hat{\beta}$. Thus, we have
\begin{align*}
2p\hat{\beta}^TA\sum_{i}k_i\bar{q}_i-2p\hat{\beta}^Tq &=2p\hat{\beta}^T(Aq-pA\hat{\beta})-2p\lambda_0\hat{\beta}^Tq\\
&\equiv 2p\hat{\beta}^T(Aq-\lambda_0q)\\
&\equiv 0~({\rm mod}~p^2).
\end{align*}
The last congruence follows, since $Aq-\lambda_0q\equiv 0$ (mod $p$). This completes the proof.
\end{proof}

\section{Proof of Theorem~\ref{Main1}}

In this section, we present the proof of Theorem~\ref{Main1}. Let $Q\in{O_n(\mathbb{Q})}$ with level $\ell$. Our strategy is to show that if $\Delta_A$ is odd and square-free, then $Q^TAQ=B\in {S_n(\mathbb{Z})}$ implies $\ell(Q)=1$. Before doing so, we need several lemmas below.

\begin{lemma}\label{corank1}
Suppose $\rank_p(\lambda_0I-A)=n-1$. Then under the conditions of Lemma~\ref{LLX}, we have $p^2|\det(\lambda_0I-A)$.
\end{lemma}

\begin{proof}
Let $\alpha:=(Aq-\lambda_{0} q)/p$ be an integral vector. Then it follows from Lemma~\ref{LLX} that $q^T\alpha\equiv 0$ (mod $p$). Thus, we have $q^T[A-\lambda_0I,\alpha]=0$ over $\mathbb{F}_p$. Note that $q\neq 0$. It follows that $\rank_p([A-\lambda_0I,\alpha])\leq n-1$. By the assumption that $\rank_p(\lambda_0I-A)=n-1$, we get $\rank_p([A-\lambda_0I,\alpha])= n-1$. Thus, $\alpha$ can be expressed as the linear combinations of the column vectors of $A-\lambda_0I$, i.e., there exists an integral vector $x$ such that $(A-\lambda_0I)x\equiv (Aq-\lambda_{0} q)/p$ (mod $p$). Then by Lemma~\ref{Square}, the lemma follows immediately.
\end{proof}

\begin{lemma}\label{corank2}
Let $B\in M_n(\mathbb{Z})$ and $p$ be a prime number. If $\rank_p(B)=n-k$, then $p^k|\det(B)$.
\end{lemma}

\begin{proof}
Since $\rank_p(B)=n-k$, there exists an $n$ by $n$ matrix $U$ and an $n-k$ by $n$ matrix $V$ such that
\[UB\equiv\begin{bmatrix}V\\0\end{bmatrix}\pmod{p},\quad\det(U)\not\equiv0\pmod{p},\]
or
\[UB=\begin{bmatrix}V\\pW\end{bmatrix},\quad p\nmid\det(U),\]
for some $k$ by $n$ matrix $W$. Therefore, $p^k|\det(U)\det(B)$. Since $p\nmid\det(U)$, we get that $p^k|\det(B)$.
\end{proof}

\begin{lemma}\label{decomp}
Let $p$ be any prime factor of $\Delta_A$. If $p^2\nmid\Delta_A$, we have $\phi(x)=(x-\lambda_0)^2\varphi(x)$ over $\mathbb{F}_p$ for some $\lambda_0$ and square-free $\varphi(x)$, where $p\nmid\varphi(\lambda_0)$.
\end{lemma}

\begin{proof}
Using the notations in Section \ref{secDiscr}, that is, let
\begin{align*}
\phi(x)&=a_0x^n+a_1x^{n-1}+\cdots+a_{n-1}x+a_n,\\
u(x)&=u_0x^{n-2}+u_1x^{n-3}+\cdots+u_{n-3}x+u_{n-2},\\
v(x)&=v_0x^{n-1}+v_1x^{n-2}+\cdots+v_{n-2}x+v_{n-1},\\
\end{align*}
\[M=\begin{bmatrix}
a_0&a_1&\cdots&\cdots&\cdots&a_n\\
&a_0&a_1&\cdots&\cdots&\cdots&a_n\\
&&\cdots&\cdots&\cdots&\cdots\\
&&&a_0&a_1&\cdots&\cdots&\cdots&a_n\\
na_0&(n-1)a_1&\cdots&\cdots&a_{n-1}\\
&na_0&(n-1)a_1&\cdots&\cdots&a_{n-1}\\
&&\cdots&\cdots&\cdots&\cdots\\
&&&&na_0&(n-1)a_1&\cdots&\cdots&a_{n-1}
\end{bmatrix}\]
and
\begin{equation}\label{KK}
\eta=(u_0,u_1,\cdots,u_{n-2},-v_0,-v_1,\cdots,-v_{n-1})^T,
\end{equation}
then $\Delta_A=\pm\det(M)$. Furthermore, $M^T\eta=0$ is equivalent to $u\phi=v\phi'$ over $\mathbb{F}_p$. By Lemma~\ref{corank2}, we have $\dim\ker_p(M)=(2n-1)-\rank_p(M)=1$. By Lemma~\ref{XS}, we can assume that $\phi(x)=g(x)^2\varphi(x)$, where $\deg(g)\ge1$. If $\deg(g)\ge2$, let
\begin{align*}
u_1(x)&=2g'(x)\varphi(x)+g(x)\varphi'(x),&u_2(x)&=xu_1(x),\\
v_1(x)&=g(x)\varphi(x),&v_2(x)&=xv_1(x).
\end{align*}
Then $u_1,v_1$ and $u_2,v_2$ are both solutions to $u\phi=v\phi'$, over $\mathbb{F}_p$. It is easy to verify that $\deg(u_i)<\deg(\phi')$ and $\deg(v_i)<\deg(\phi)$, for $i=1,2$. Let $\eta_1$ (resp. $\eta_2$) be the vector obtained from Eq.~(\ref{KK}) by
replacing the coefficients of $u$ and $v$ with that of $u_1$ and $v_1$ (resp. $u_2$ and $v_2$). It is easy to see that $\eta_1$ and $\eta_2$  are linearly independent. Thus, $\dim\ker_p(M)$ is at least 2, contradicting that $\dim\ker_p(M)=1$.

Now, let $g(x)=x-\lambda_0$. If $p|\varphi(\lambda_0)$, we can write $\varphi(x)=(x-\lambda_0)\hat{\varphi}(x)$. Let
\begin{align*}
u_1(x)&=3\hat{\varphi}(x)+(x-\lambda_0)\hat{\varphi}'(x),&u_2(x)&=xu_1(x),\\
v_1(x)&=(x-\lambda_0)\hat{\varphi}(x),&v_2(x)&=xv_1(x).
\end{align*}
Then using the similar arguments, $u_1,v_1$ and $u_2,v_2$ are both solutions to $u\phi=v\phi'$, and we will get a contradiction as before. Therefore, $\phi(x)=(x-\lambda_0)^2\varphi(x)$ over $\mathbb{F}_p$ for some $\lambda_0$ and square-free $\varphi(x)$, where $p\nmid\varphi(\lambda_0)$.
\end{proof}

\begin{lemma}\label{cogEqSolEx}
Let $p$ be an odd prime. Consider the equation
\begin{equation}\label{polyRes}
u(x)\phi(x)\equiv v(x)\phi'(x)\pmod{p^2}.
\end{equation}
Under the conditions of Lemma~\ref{decomp}, Eq. \eqref{polyRes} has a solution such that $u,v\not\equiv0\pmod{p}$ and that $\deg(v)<\deg(\phi)$ and $\deg(u)<\deg(\phi')$ if and only if $p^2|\det(\lambda_0I-A)$.
\end{lemma}

\begin{proof}
 By Lemma~\ref{decomp}, we can write
\[\phi(x)=(x-\lambda_0)^2\varphi(x)+pr(x),~{\rm over}~\mathbb{Z}.\]
where $\deg(r)\le \deg(\phi)-1$. Then $\phi'=(x-\lambda_0)(2\varphi+(x-\lambda_0)\varphi')+pr'$. (Notice: Here and below, for simplicity, we write $f$ and $f(x)$ interchangeably if no confusion arises).

First, we prove the necessary part of the lemma. Assume that such $u$ and $v$ exist. Changing the modulus of Eq.~\eqref{polyRes} to $p$, we have
\[u(x)(x-\lambda_0)^2\varphi(x)\equiv v(x)(x-\lambda_0)(2\varphi(x)+(x-\lambda_0)\varphi'(x))\pmod{p},\]
which means $(x-\lambda_0)\varphi(x)|v(x)(2\varphi(x)+(x-\lambda_0)\varphi'(x))$ over $\mathbb{F}_p$. Since $\varphi(x)$ is square-free (that is, $(\varphi(x),\varphi'(x))=1$) and $(x-\lambda_0)\nmid\varphi(x)$, we have $(x-\lambda_0)\varphi(x)|v(x)$. By the degree restriction, we have $\deg(v)\leq \deg((x-\lambda_0)\varphi(x))$. So we can write
\begin{equation}\label{solPat}
\begin{split}
v(x)&=k(x-\lambda_0)\varphi(x)+pg(x),\\
u(x)&=k(2\varphi(x)+(x-\lambda_0)\varphi'(x))+pf(x)
\end{split}
\end{equation}
for some scalar $k$ and polynomials $f$ and $g$, where $p\nmid k$, $\deg(f)<\deg(\phi)-1$ and $\deg(g)<\deg(\phi)$. Without loss of generality, we can assume that $k=1$. Substitute $u$ and $v$ into Eq.~(\ref{polyRes}), we get
\begin{align*}
u\phi&\equiv (x-\lambda_0)^2(2\varphi+(x-\lambda_0)\varphi')\varphi\\
&\quad+p[(2\varphi+(x-\lambda_0)\varphi')r+(x-\lambda_0)^2\varphi f]~({\rm mod}~p^2),\\
v\phi'&\equiv (x-\lambda_0)^2(2\varphi+(x-\lambda_0)\varphi')\varphi\\
&\quad+p[(x-\lambda_0)r'\varphi+(x-\lambda_0)(2\varphi+(x-\lambda_0)\varphi')g]~({\rm mod}~p^2).
\end{align*}
Therefore, we have
$$(2\varphi+ (x-\lambda_0)\varphi')r+ (x-\lambda_0)^2\varphi f= (x-\lambda_0)r'\varphi+ (x-\lambda_0)(2\varphi+ (x-\lambda_0)\varphi')g$$
over $\mathbb{F}_p$. Rewrite it as
\[[2r+ (x-\lambda_0)^2f- (x-\lambda_0)r-2 (x-\lambda_0)g]\varphi=[ (x-\lambda_0)^2g- (x-\lambda_0)r]\varphi'.\]
Therefore, there exists a polynomial $\Phi$ such that
\begin{equation}\label{fgEq}
\begin{split}
2r+ (x-\lambda_0)^2f- (x-\lambda_0)r'-2 (x-\lambda_0)g&=\Phi\varphi',\\
(x-\lambda_0)( (x-\lambda_0)g-r)&=\Phi\varphi.
\end{split}
\end{equation}
The second equation implies that $ (x-\lambda_0)|\Phi$. Then it follows from the first equation that $ (x-\lambda_0)|r(x)$ over $\mathbb{F}_p$. Thus, we get that $\det(\lambda_0I-A)=\phi(\lambda_0)=pr(\lambda_0)$ is divisible by $p^2$.

Now we prove the reverse. Since $r(\lambda_0)=\phi(\lambda_0)/p=\det(\lambda_0I-A)/p$ is divisible by $p$, we can write
\begin{equation}\label{remainder}
r(x)=(x-\lambda_0)\hat{r}(x)+pr_0(x).
\end{equation}
 Let
\[f=(r'+2\mu\varphi)/(x-\lambda_0)+\mu\varphi',g=\hat{r}+\mu\varphi,\]
where $\mu=-[2\varphi(\lambda_0)]^{-1}r'(\lambda_0)$ over $\mathbb{F}_p$. Then define $u$ and $v$ by Eqs.~\eqref{solPat}. Since
\begin{align*}
\deg(f)&\le\max\{\deg(r')-1,\deg(\varphi)-1,\deg(\varphi')\}=\deg(\phi)-3,\\
\deg(g)&\le\max\{\deg(r)-1,\deg(\varphi)\}=\deg(\phi)-2,
\end{align*}
it is easy to verify that $\deg(v)<\deg(\phi)$ and $\deg(u)<\deg(\phi')$. Thus, the degree restrictions are satisfied.

Moreover, by the discussions above, we only need to verify Eqs.~\eqref{fgEq} hold. Actually, we have
\begin{align*}
&\quad 2r+(x-\lambda_0)^2f-(x-\lambda_0)r'-2(x-\lambda_0)g\\
&=2(x-\lambda_0)\hat{r}+(x-\lambda_0)(r'+2\mu\varphi)+\mu(x-\lambda_0)^2\varphi'-(x-\lambda_0)r'-2(x-\lambda_0)(\hat{r}+\mu\varphi)\\
&=(\mu(x-\lambda_0)^2)\varphi',
\end{align*}
and
\[(x-\lambda_0)((x-\lambda_0)g-r)=(x-\lambda_0)^2(g-\hat{r})=(\mu(x-\lambda_0)^2)\varphi\]
over $\mathbb{F}_p$, where we have used Eq.~\eqref{remainder}. Therefore, the $u$ and $v$ above satisfy all the conditions in the lemma. The proof is complete.
\end{proof}

\begin{lemma}\label{pnmidl}
Let $Q$ be a rational orthogonal matrix with level $\ell$ such that $Q^TAQ\in{S_n(\mathbb{Z})}$. Let $p$ be an odd prime factor of $\Delta_A$. If $p^2\nmid\Delta_A$, then $p\nmid\ell$.
\end{lemma}

\begin{proof}
We prove the lemma by contradiction. Suppose that $p|\ell$. By Lemma~\ref{decomp}, there exists a $\lambda_0$ such that $\phi(x)=(x-\lambda_0)^2\varphi(x)$ over $\mathbb{F}_p$, where $\varphi$ is square-free and $p\nmid\varphi(\lambda_0)$. Thus $\rank_p(\lambda_0I-A)<n$. If $\rank_p(\lambda_0I-A)=n-1$, by Lemma~\ref{corank1}, $p^2|\det(\lambda_0I-A)$. If $\rank_p(\lambda_0I-A)<n-1$, by Lemma~\ref{corank2}, we also have $p^2|\det(\lambda_0I-A)$. Then by Lemma~\ref{cogEqSolEx}, there exist $u(x)$ and $v(x)$ with $\deg(v)<\deg(\phi)$ and $\deg(u)<\deg(\phi')$ such that $u,v\not\equiv0\pmod{p}$ and that $u\phi\equiv v\phi'\pmod{p^2}$. Equivalently, using the notations in Lemma~\ref{decomp}, we have
\[M^T\eta\equiv0\pmod{p^2},\quad \eta\not\equiv0\pmod{p}.\]
By Lemma~\ref{Square}, we have $p^2|d_n(M)$. Thus, $p^2|\det(M)=\pm\Delta_A$; it is a contradiction.
\end{proof}

\begin{rem}
As a by-product of the above lemma, we get that $\rank_p(A)=n-1$ always holds for any prime $p|\Delta_A$ (notice that $p$ do not need to be a divisor of $\ell$) if $\Delta_A$ is odd and square-free. This is because if $\rank_p(\lambda_0I-A)<n-1$, then $p^2|\det(\lambda_0I-A)$, which implies $p^2|\Delta_A$ according to Lemma~\ref{pnmidl} and we get a contradiction since $\Delta_A$ is square-free.
\end{rem}


Now, we are ready to present the proof of Theorem~\ref{Main1}.

\begin{proof}[\textup{\textbf{Proof of Theorem~\ref{Main1}}}]
We prove the theorem by contradiction. Let $Q\in{O_n(\mathbb{Q})}$ with level $\ell$ such that $Q^TAQ\in S_n(\mathbb{Z})$. Assume on the contrary that $\ell\neq 1$. Let $p$ be any prime divisor of $\ell$. Then $p|\Delta_A$ and $p$ is odd. By the assumption, we have $p^2\nmid \Delta_A$. Thus, according to Lemma~\ref{pnmidl},
$p\nmid \ell$; a contradiction. Thus $\ell=1$ and $Q$ is a signed permutation matrix. The proof is complete.
\end{proof}

\section{Proof of Theorem~\ref{Main2}}

In this section, we present the proof of Theorem~\ref{Main2}. Before doing so, we need the following lemma.

\begin{lemma}\label{Even}Let $A$ be the adjacency matrix of a graph $G$. Then $\Delta_A$ is always even.
\end{lemma}
\begin{proof} Let $\phi(x) = x^n +c_{1}x^{n-1} +\cdots+
c_{n-1}x + c_n$ be the characteristic polynomial of graph $G$. Then by Sachs' Coefficients Theorem (see e.g.~\cite{CDS}), we have
\begin{equation}\label{KLM}
c_i = \sum\limits_{H \in {\cal H}_i}(-1)^{p(H)}2^{c(H)},
\end{equation}
where ${\cal H}_i$ is the set of elementary graphs (i.e., graphs in which every component is either $K_2$ or a cycle) with $i$ vertices in $G$, $p(H)$ the number of
components of $H$ and $c(H)$ the number of cycles in $H$. It follows from Eq.~(\ref{KLM}) that $c_i$ is even when $i$ is odd.

If $n$ is even, we have
\begin{align*}
\phi(x)&\equiv x^n+c_2x^{n-2}+\cdots+c_{n-2}x^2+c_n\\
&\equiv (x^{n/2}+c_2x^{(n-2)/2}+\cdots+c_{n-2}x+c_n)^2\pmod{2}.
\end{align*}
If $n$ is odd, we have
\begin{align*}
\phi(x)&\equiv x^n+c_2x^{n-2}+\cdots+c_{n-3}x^3+c_{n-1}x\\
&\equiv x(x^{(n-1)/2}+c_2x^{(n-3)/2}+\cdots+c_{n-3}x+c_{n-1})^2\pmod{2}.
\end{align*}
In both cases, we get that $\phi$ has a multiple factor over $\mathbb{F}_2$. It follows from Theorem~\ref{divisor} that $2|\Delta_A$.
\end{proof}

\begin{proof}[\textup{\textbf{Proof of Theorem~\ref{Main2}.}}]
Let $Q \in {\cal Q}_G$ with level $\ell$. We show that $\ell=1$. Suppose on the contrary that $\ell\neq 1$. Then $\ell$ has a prime divisor $p$. We distinguish the following two cases.

\noindent
\textbf{Case 1.} The prime $p=2$. It follows from Lemma~\ref{Even} that $\Delta_A$ is even. Notice that $d=\gcd(\frac{\det(W)}{2^{\lfloor n/2\rfloor}},\Delta_A)$ is odd by the assumption. It follows that $2^{\lfloor n/2\rfloor+1}\nmid\det(W)$ (otherwise, $d$ would be even). Thus, by Theorem~\ref{LL1} we get that $2\nmid \ell$; this is a contradiction.

\noindent
\textbf{Case 2.} The prime $p$ is odd. According to Lemma~\ref{ellRes}, we have $p|\Delta_A$. By Lemma~\ref{L}, we have $p|\det(W)$, and hence $p|\frac{\det(W)}{2^{\lfloor n/2\rfloor}}$. It follows that $p|d$. Since $d$ is square-free, we have either $p^2\nmid \det(W)$ or $p^2\nmid \Delta_A$. If the former holds, we have $p\nmid \ell$ according to Theorem~\ref{LL}. If the latter holds, we have $p\nmid \ell$ according to Lemma~\ref{pnmidl}. In both cases, we have $p\nmid \ell$, which contradicts the fact that $p$ is a divisor of $\ell$.

Combining Cases 1 and 2, we have $\ell=1$ and $Q$ is a permutation matrix. This completes the proof.
\end{proof}

\section{Numerical results}

In this section, we shall provide some numerical results. First, we give some examples to illustrate the method of the paper, and then we shall conduct numerical experiments to compare Theorems~\ref{Main} and \ref{Main2}. The matrices $A\in{S_n(\mathbb{Z})}$ and the graph $G$ are randomly generated by a personal computer.
All the computations were carried out by using Mathematica 8.0. \\

\vskip 1\baselineskip\noindent
\textbf{Example 1.}
$$A=\left[\begin{array}{rrrrrrrr}
1& 0& 1& -2& -1& 1& -1& -2\\
0&-1& -1& 0& 0& -1& 0& 1\\
1& -1& 0&  1& 0& -2& -2& 1\\
-2& 0& 1& 0& -1& 0& -2& 0\\
-1& 0& 0& -1&  3& -1&1& -2\\
1&-1& -2& 0& -1& 0& -1& -2\\
-1& 0& -2& -2& 1& -1&   0& 0\\
-2& 1& 1& 0& -2&-2& 0&0
\end{array}\right]_{8\times 8}.
$$

It can be computed that $\Delta_A=23\times 7309\times79967\times300191\times 146237798587879.$ It follows from Theorem~\ref{Main1}
that $Q^TAQ\in S_n(\mathbb{Z})$ with $Q\in{O_n(\mathbb{Q})}$ implies that $Q$ is a signed permutation matrix.

\vskip 1\baselineskip\noindent
\textbf{Example 2. }

$$A=\left[\begin{array}{rrrrrrrrrr}
1 & 1 & -1 & 1 & 0 & -1 & -1 & 1 & 0 & 1 \\
 1 & 0 & 1 & 0 & 0 & 0 & 0 & 1 & 1 & -1 \\
 -1 & 1 & -1 & 0 & 1 & -1 & 0 & -1 & 0 & 0 \\
 1 & 0 & 0 & 0 & 1 & 0 & 0 & 1 & 0 & 1 \\
 0 & 0 & 1 & 1 & 1 & 0 & 1 & 0 & 0 & 0 \\
 -1 & 0 & -1 & 0 & 0 & 0 & 1 & 0 & 0 & -1 \\
 -1 & 0 & 0 & 0 & 1 & 1 & 0 & 1 & -1 & -1 \\
 1 & 1 & -1 & 1 & 0 & 0 & 1 & 0 & 1 & 1 \\
 0 & 1 & 0 & 0 & 0 & 0 & -1 & 1 & 0 & -1 \\
 1 & -1 & 0 & 1 & 0 & -1 & -1 & 1 & -1 & 0
\end{array}\right]_{10\times 10}.
$$

It can be computed that $\Delta_A=3^ 2\times 761\times 26561\times 36102323417\times 29304766290781$.
Thus, $\Delta_A$ is not square-free.
Let
$$Q=\left[\begin{array}{rrrrrrrrrr}
 0 & 0 & 0 & 0 & 0 & 0 & 1 & 0 & 0 & 0 \\
 0 & 0 & 0 & 0 & 0 & 0 & 0 & 1 & 0 & 0 \\
 0 & 0 & 0 & 0 & 0 & 0 & 0 & 0 & 1 & 0 \\
 \frac{2}{3} & -\frac{1}{3} & -\frac{1}{3} & \frac{1}{3} & \frac{1}{3} & \
\frac{1}{3} & 0 & 0 & 0 & 0 \\
 0 & 0 & 0 & 0 & 0 & 0 & 0 & 0 & 0 & 1 \\
 \frac{1}{3} & \frac{1}{3} & \frac{1}{3} & \frac{2}{3} & -\frac{1}{3} & \
-\frac{1}{3} & 0 & 0 & 0 & 0 \\
 \frac{1}{3} & \frac{1}{3} & \frac{1}{3} & -\frac{1}{3} & \frac{2}{3} & \
-\frac{1}{3} & 0 & 0 & 0 & 0 \\
 -\frac{1}{3} & \frac{2}{3} & -\frac{1}{3} & \frac{1}{3} & \frac{1}{3} & \
\frac{1}{3} & 0 & 0 & 0 & 0 \\
 -\frac{1}{3} & -\frac{1}{3} & \frac{2}{3} & \frac{1}{3} & \frac{1}{3} & \
\frac{1}{3} & 0 & 0 & 0 & 0 \\
 \frac{1}{3} & \frac{1}{3} & \frac{1}{3} & -\frac{1}{3} & -\frac{1}{3} & \
\frac{2}{3} & 0 & 0 & 0 & 0
\end{array}\right]_{10\times 10}.
$$
We have
$$Q^TAQ=\left[\begin{array}{rrrrrrrrrr}
 0 & 1 & -1 & 0 & 0 & 0 & 0 & -1 & 0 & 1 \\
 1 & 0 & 0 & 0 & 1 & 0 & 0 & 0 & -1 & 0 \\
 -1 & 0 & -2 & 1 & 0 & -1 & -1 & 0 & 0 & 0 \\
 0 & 0 & 1 & 0 & 1 & 0 & 0 & 1 & -1 & 0 \\
 0 & 1 & 0 & 1 & 0 & 0 & 0 & 1 & 0 & 1 \\
 0 & 0 & -1 & 0 & 0 & 2 & 2 & 0 & 0 & 0 \\
 0 & 0 & -1 & 0 & 0 & 2 & 1 & 1 & -1 & 0 \\
 -1 & 0 & 0 & 1 & 1 & 0 & 1 & 0 & 1 & 0 \\
 0 & -1 & 0 & -1 & 0 & 0 & -1 & 1 & -1 & 1 \\
 1 & 0 & 0 & 0 & 1 & 0 & 0 & 0 & 1 & 1
\end{array}\right]_{10\times 10}.
$$
Thus, $Q^TAQ\in{S_n{(\mathbb{Z})}}$. This shows that Theorem~\ref{Main1} is the best possible, in the sense that
if $\Delta_A$ is allowed to have an odd prime factor with exponent larger than one, then it is no longer true that $Q^TAQ\in S_n(\mathbb{Z})$ with $Q\in{O_n(\mathbb{Q})}$ implies that $Q$ is a signed permutation matrix.

\vskip 1\baselineskip\noindent
\textbf{Example 3. }Let the adjacency matrix of graph $G$ be given as following:
$$A=\left[\begin{array}{rrrrrrrrrrrr}
 0 & 0 & 0 & 1 & 0 & 1 & 1 & 0 & 1 & 1 & 1 & 0 \\
 0 & 0 & 1 & 1 & 0 & 1 & 0 & 0 & 0 & 1 & 1 & 1 \\
 0 & 1 & 0 & 0 & 0 & 1 & 0 & 0 & 0 & 1 & 0 & 1 \\
 1 & 1 & 0 & 0 & 0 & 1 & 0 & 1 & 1 & 1 & 1 & 1 \\
 0 & 0 & 0 & 0 & 0 & 1 & 1 & 0 & 0 & 0 & 0 & 1 \\
 1 & 1 & 1 & 1 & 1 & 0 & 1 & 0 & 0 & 1 & 0 & 0 \\
 1 & 0 & 0 & 0 & 1 & 1 & 0 & 0 & 0 & 1 & 1 & 0 \\
 0 & 0 & 0 & 1 & 0 & 0 & 0 & 0 & 0 & 0 & 0 & 1 \\
 1 & 0 & 0 & 1 & 0 & 0 & 0 & 0 & 0 & 0 & 0 & 1 \\
 1 & 1 & 1 & 1 & 0 & 1 & 1 & 0 & 0 & 0 & 1 & 0 \\
 1 & 1 & 0 & 1 & 0 & 0 & 1 & 0 & 0 & 1 & 0 & 0 \\
 0 & 1 & 1 & 1 & 1 & 0 & 0 & 1 & 1 & 0 & 0 & 0
\end{array}\right]_{12\times 12}.
$$
It can be computed that $\Delta_A=2^{ 12}\times 5^3\times 23\times 91502697363972395639457912547$ and
$\det(W)=2^6\times 5\times 13^ 2\times 569759.$ Theorem~\ref{Main} cannot be applied, since $\det(W)/2^6$ is not square-free.
However, we have $d=\gcd(\frac{\det(W)}{2^{\lfloor\frac{n}{2}\rfloor}},\Delta_A)=5$, which is odd and square-free. It follows from Theorem~\ref{Main2} that $G$ is DGS.

\begin{rem} In Example 3, we do not need to factor $\Delta_A$ and $\det(W)$, which are usually quite large when $n$ is large.
Instead, we use Euclidean Algorithm to compute $\gcd(\frac{\det(W)}{2^{\lfloor\frac{n}{2}\rfloor}},\Delta_A)$, which is much faster than factoring $\Delta_A$ and $\det(W)$.
\end{rem}

In the remaining part of this section, we shall conduct numerical experiments to compare Theorems~\ref{Main} and \ref{Main2}. Define
 $${\cal F}_n=\{G|\frac{\det(W)}{2^{\lfloor\frac{n}{2}\rfloor}}~{\rm is~odd~and~square-free}\}$$ and
 $${\cal F}'_n=\{G|\gcd(\frac{\det(W)}{2^{\lfloor\frac{n}{2}\rfloor}},\Delta_A)~{\rm is~odd~and~square-free}\}.$$

 It is apparent that ${\cal{F}}_n\subset {\cal{F'}}_n$. We have performed a series of numerical experiments to see how large the family of graphs ${\cal F}'_n $ is, compared with ${\cal F}_n $. The method is similar to that in~\cite{W4}. The graphs are generated randomly and independently from
the probability space ${\cal G}(n,\frac{1}{2})$ (see e.g. \cite{1}). At each time, we generated $1,000$
graphs randomly, and counted the number of graphs that are in ${\cal F'}_n$. Table 1
records one of such experiments (note the results may vary slightly at each
run of the algorithm; the number of graphs that are in ${\cal F}_n$ was copied from \cite{W4}). The first column is the order $n$ of the graphs generated
varying from $10$ to $80$. The second (resp. the third) column records the number of graphs that belongs to ${\cal F}'_n$ (resp. ${\cal F}_n$) among the randomly generated $1,000$ graphs. The fourth column is the ratio of the number of graphs in ${\cal F'}_n$ and that in ${\cal F}_n$.

\begin{table}[!htp]
\centering

\caption{\textbf{Fractions of Graphs in} ${\cal F'}_n$ and ${\cal F}_n$
 \label{table2.1}}
\begin{tabular}{c|c|c|c}
\hline $n$ & $\#$ (${\cal F}'_n$)&$\#$ (${\cal F}_n$)& $\#({\cal F}'_n)/\# ({\cal F}_n)$\\\hline
\hline
$10$ & 290 & 211 & 1.37441 \\\hline
$15$ & 270 & 201 & 1.34328 \\\hline
$20$ & 279 & 213 & 1.30986 \\\hline
$25$ & 277 & 216 & 1.28241 \\\hline
$30$ & 294 & 233 & 1.2618  \\\hline
$35$ & 275 & 229 & 1.20087 \\\hline
$40$ & 281 & 198 & 1.41919 \\\hline
$45$ & 263 & 202 & 1.30198 \\\hline
$50$ & 261 & 204 & 1.27941 \\\hline
$60$ & 271 & *   & * \\\hline
$70$ & 281 & *   & * \\\hline
$80$ & $277 $ &*& * \\\hline
\end{tabular}

\end{table}
It can be seen from Table 1 that the number of graphs in ${\cal F'}_n$ is larger than that in ${\cal F}_n$ by approximately $30\%$ on average. This shows that
Theorem~\ref{Main2} is more applicable than Theorem~\ref{Main}.
\section{Conclusions and future work}

Motivated by the generalized spectral characterization of graphs, in this paper, we have considered the problem
of when $Q^TAQ\in{S_n(\mathbb{Z})}$ with $Q\in {O_n(\mathbb{Q})}$ implies that $Q$ is a signed permutation matrix.
A surprisingly simple answer was provided in terms of whether the discriminant of $A$ is odd and square-free.
Combined with previous work, we gave a new and efficient method for a graph being determined by its generalized spectrum, which
has several advantages over the previous one.

However, we believe that our work in the paper can be extended in several directions.

\begin{enumerate}
\item We only dealt with a special case of Problem 1 (i.e., Problem 2) by providing a sufficient condition, under which all matrices $Q\in {O_n(\mathbb{Q})}$ such that $Q^TAQ\in{S_n(\mathbb{Z})}$ are signed permutation matrices. The general situation still needs further investigations.

    \item In attacking Problem 2, we only consider the case that $\Delta_A$ is odd (notice that we proved Lemma~\ref{cogEqSolEx}, which is a key part of the proof of Theorem~\ref{Main1}, only for odd primes). As shown by Lemma~\ref{Even}, $\Delta_A$ is always even, when $A$ is the adjacency matrix of a graph (or more generally, when all the diagonal entries of $A$ are even). This means the conditions of Theorem~\ref{Main1} are not satisfied when $A$ is the adjacency matrix of a graph. Thus, to deal with the situation that $\Delta_A$ is even, we need to develop new method for the prime $p=2$.

\item It is well-known that whether the discriminant $\Delta(f)$ of a polynomial $f(x)\in{\mathbb{Z}[x]}$ is zero is closely related to whether $f$ has repeated roots over $\mathbb{C}$. Moreover,
if $f(x)$ is a monic irreducible polynomial over $\mathbb{Q}$ and $\Delta(f)$ is square-free, the ring of algebraic integers $\mathbb{O}_K$ in $\mathbb{K}=\mathbb{Q}[x]/(f(x))$ can be generated by a single element, i.e., $\mathbb{O}_K=\mathbb{Q}[\alpha]$. In view of this, several authors have investigated the density of polynomials with square-free discriminates (in a precise sense). In a recent break-through work, Bhargava et al.~\cite{BSW} proved that for monic polynomials of any fixed degree $n$, the density of polynomials with square-free discriminates is as expected, namely, it is the Euler product over all primes $p$ of the probabilities that is square-free at $p$, i.e,
$\Pi_p a_p$, where $p$ runs over prime numbers and $a_p$ denotes the probability that $\Delta_A$ is not divisible by $p^2$. These probability have been computed by Ash et. al.~\cite{ABZ} as follows:
\begin{eqnarray*}
a_p=\left\{
\begin{array}{cc}
\frac{1}{2}&p=2,n\geq 2;\\
1-\frac{1}{p^2}&p>2,n=2;\\
1-\frac{2}{p^2}+\frac{1}{p^3}&p>2,n=3;\\
1-\frac{1}{p}+\frac{(p-10^2(1-(-p)^{-n+1})}{p^2(p+1)}&p>2,n\geq 4.
\end{array}\right.
\end{eqnarray*}
 The density tends to approximately $35.8232\%$ as $n$ goes to infinity.
  It is an interesting future work to determine the density of graphs with square-free discriminants (for odd primes), which in turn, would give rise to the density of graphs determined by their generalized spectra.

\end{enumerate}


\begin{thebibliography}{22}

\bibitem{ABZ} A. Ash, J. Brakenhoff and T. Zarrabi, Equality of polynomial and field discriminants, Exper. Math., 16 (2007) 367-374.

\bibitem{BSW}M. Bhargava, A. Shankar and  X. Wang, Squarefree values of polynomial discriminants, In preparation.

\bibitem{1}
 B. Bollob\'{a}s, Modern Graph Theory, Springer-Verlag, NewYork, 2002.



\bibitem{CDS}
D. M. Cvetkovi\'{c}, M. Doob, H. Sachs, Spectra of Graphs, Academic Press,
NewYork, 1982.


\bibitem{DH}
E. R. van Dam, W. H. Haemers, Which graphs are determined by their spectrum?
Linear Algebra Appl., 373 (2003) 241-272.

\bibitem{DH1}
 E. R. van Dam, W. H. Haemers, Developments on spectral characterizations
of graphs, Discrete Mathematics, 309 (2009) 576-586.


\bibitem{Lang}
 S. Lang, Algebra, Springer-Verlag, New York, 2002.

\bibitem{F} M. Fisher, On hearing the shape of a drum, J. Combin. Theory, 1 (1966) 105-125.


\bibitem{Fr}S. Friedland, Rational orthogonal similarity of rational symmetric matrices, Linear Algebra Appl., 192 (1993) 109-114.

\bibitem{Kac} M. Kac, Can one hear the shape of a drum? Amer. Math. Monthly, 73 (1966) 1-23.

\bibitem{GM} C.D. Godsil, B.D. McKay, Constructing cospectral graphs, Aequation Mathematicae,
25 (1982) 257-268.

\bibitem{G} C.D. Godsil, Controllable subsets in graphs, Annals of Combinatorics, 16
(2012) 733-744.

\bibitem{OT}S. O'Rourke and B. Touri, On a conjecture of Godsil concerning controllable random graphs, http://arxiv.org/abs/1511.05080

\bibitem{15}
A. Schrijver, The Theory of Linear and Integer Programming, John Wiley $\&$ Sons, 1998.

\bibitem{W1}
 W. Wang, C. X. Xu, A sufficient condition for a family of graphs being determined
by their generalized spectra, European J. Combin., 27 (2006) 826-840.

\bibitem{W2}
 W. Wang, C.X. Xu, An excluding algorithm for testing whether a family of
graphs are determined by their generalized spectra, Linear Algebra and its
Appl., 418 (2006) 62-74.


\bibitem{W3} W. Wang, Generalized spectral characterization revisited, The Electronic J.
Combin., 20 (4) (2013), $\sharp$ P4.


\bibitem{W4} W. Wang, A simple arithmetic criterion for graphs being
determined by their generalized spectra, J. Combin. Theory, Ser. B (2016),
 http://dx.doi.org/10.1016/j.jctb.2016.07.004.



\end{thebibliography}
\end{document}